\def\MIHALIS{1} 

\documentclass[reqno,a4paper,12pt]{amsart} 

\usepackage{amssymb,bm,mathrsfs,mathtools,booktabs}

\usepackage{enumerate}
\usepackage{tikz}
\usepackage[centering,width=6in]{geometry}
\usepackage{graphics,verbatim}
\usepackage{graphicx}
\usepackage{csquotes}
\usepackage{enumitem}
\usepackage{thmtools} 
\usepackage{thm-restate}
\usepackage{dsfont}

\ifdefined\MIHALIS  
\usepackage{txfonts,pxfonts,tikz} 
\usepackage{tgschola}
\usepackage[T1]{fontenc}
\usepackage[allcolors=blue,allbordercolors=blue,pdfborderstyle={/S/U/W 1}]{hyperref}
\usepackage[ocgcolorlinks]{ocgx2}
\usepackage{bbm}

\ifdefined\SMART 
\usepackage[paperwidth=12cm,paperheight=24cm,left=5mm,right=5mm,top=10mm,bottom=5mm]{geometry}
\else
\calclayout
\fi

\fi 

\numberwithin{equation}{section}

\newcommand{\R}{{\mathbb R}}
\newcommand{\Z}{{\mathbb Z}}

\newcommand{\N}{{\mathbb N}}

\newcommand{\C}{{\mathcal C}}

\newcommand{\Abs}[1]{{\left|{#1}\right|}}

\newcommand{\Norm}[1]{{\left\|{#1}\right\|}}

\newcommand{\Set}[1]{{\left\{{#1}\right\}}}

\newcommand{\RR}{{\mathbb R}}

\newcommand{\CC}{{\mathbb C}}
\newcommand{\ZZ}{{\mathbb Z}}

\newcommand{\inner}[2]{{\langle #1, #2 \rangle}}

\newcommand{\ft}[1]{\widehat{#1}}

\newcounter{rem}
\newtheorem{theorem}{Theorem}[section]
\newtheorem{lemma}[theorem]{Lemma}
\newtheorem{proposition}[theorem]{Proposition}
\newtheorem{corollary}[theorem]{Corollary}

\newtheorem{problem}[theorem]{Problem}
\theoremstyle{remark}

\newtheorem{example}{Example}[section]
\numberwithin{equation}{section}

 \DeclareMathOperator\supp{\rm supp}

\def\N{\mathbb{N}}
\def\Z{\mathbb{Z}}
\def\R{\mathbb{R}}
\def\C{\mathbb{C}}
\def\orange{\org}
\def\org#1{\textcolor[rgb]{1.0, 0.5, 0.00}{#1}}
\def\blue{\blue}
\def\blue#1{\textcolor[rgb]{0.00,0.00,1.00}{#1}}
\def\red#1{\textcolor[rgb]{1.00,0.00,0.00}{#1}}

\title[Spectrality of factors of product spectral measures]{Spectrality of factors of  product spectral measures}

\author{Mihail N. Kolountzakis}
\address{\href{http://math.uoc.gr/en/index.html}{Department of Mathematics and Applied Mathematics}, University of Crete,\\Voutes Campus, 70013 Heraklion, Greece,\newline and \newline \href{https://ics.forth.gr/}{Institute of Computer Science}, Foundation of Research and Technology Hellas, N. Plastira 100, Vassilika Vouton, 700 13, Heraklion, Greece}
\email{kolount@uoc.gr}

\author{Chun-Kit Lai}
\address[Chun-Kit Lai]{San Francisco State University, San Francisco, USA}
\email{cklai@sfsu.edu}

\author{Kailing Lai}
\address[Kailing Lai]{School of Mathematics and Information Science, Guangzhou University, Guangzhou, 510006, P.~R.~China,\newline and \newline
\href{http://math.uoc.gr/en/index.html}{Department of Mathematics and Applied Mathematics}, University of Crete,\\Voutes Campus, 70013 Heraklion, Greece}
\email{lkailinggl@163.com}

\author[Jinjun Li]{Jinjun Li}

 \address[Jinjun Li]{School of Mathematics and Information Science, Guangzhou University, Guangzhou, 510006, P.~R.~China}
\email{lijinjun@gzhu.edu.cn}

\subjclass[2020]{42C99}
\keywords{ Spectral measures, Product spectral set, Complementary pairs, singular measures}

\begin{document}
\begin{abstract}
We refine the method by Greenfeld and Lev for the product spectral set problem and generalize the theorem to a singular measure setting. Furthermore, we establish a new class of spectral unions of intervals for which the product spectral set question has a positive answer. More precisely,  if $A$ is a subset of the natural numbers such that $A\oplus B = \{0,1,\cdots, N-1\}$ for some $B\subset \N$ and $N>1$ then  the product measure $\mathcal{L}|_{A+[0,1]}\times \nu$ is a spectral measure (that may be singular) if and only if $\nu$ is a spectral measure. 
\end{abstract}

\date{\today}

\maketitle

\tableofcontents

\section{Introduction}

\subsection{Spectral sets and spectral measures.} Let $\Omega\subseteq\mathbb{R}^{d}$ be a bounded, measurable set of positive Lebesgue measure. We say that $\Omega$ is a \textit{spectral set} if there exists a countable set $\Lambda\subseteq\mathbb{R}^{d}$ such that the set of exponential functions
$E(\Lambda):=\{e^{2\pi i\lambda\cdot x}:\lambda\in \Lambda\}$
forms an orthogonal basis for $\textit{L}^{2}(\Omega)$. Such a set $\Lambda$ is called a \textit{spectrum }for $\Omega$. It is well-known that the classical example of a spectral set is the unit cube $[0,1]^{d}$ with a spectrum $\mathbb{Z}^{d}$. The study of spectral sets dates back to Fuglede \cite{F-JFA}, who conjectured that a set $\Omega\subset\mathbb{R}^{d}$ is a spectral set if and only if it tiles $\mathbb{R}^{d}$ by translations, i.e., there is  a discrete set $\mathcal{T}\subset\mathbb{R}^{d}$ such that $\{\Omega+t,t\in\mathcal{T}\}$ constitutes a partition of $\mathbb{R}^{d}$ up to measure zero. Fuglede's conjecture remains open in dimension one and two although it was disproved for three or higher dimensions (\cite{KM-CM}, \cite{KM-FM}, \cite{M-PAMS}, \cite{T-MRL}).

The concept of spectral sets was naturally generalized to spectral measures. {A finite positive Borel measure $\mu$ on $\mathbb{R}^{d}$ is said to be a \textit{spectral measure} if there exists a countable set $\Lambda\subset\mathbb{R}^{d}$, called a \textit{spectrum} of $\mu$, such that the set of exponential functions
$E(\Lambda)=\{e^{2\pi i\langle\lambda, x\rangle}:\lambda\in \Lambda\}$
forms an orthogonal basis for $\textit{L}^{2}(\mu)$.} It is not hard to see that there are purely atomic spectral measures, which has led to the extensive study of the Fuglede's conjecture on finite abelian groups.  The first non-atomic, singular continuous spectral measure was found by Jorgensen and Pedersen \cite{dense}. Denote by $\delta_a$ the Dirac mass at the point $a$.  Jorgensen and Pedersen \cite{dense} proved that {the standard Cantor measure with contraction ratio $1/4$ and digit sets $\{0,2\}$, which can be written as an infinite convolution of normalized Dirac probability measures:
\begin{equation}\label{jdcd}
\nu= \nu_1\ast\nu_2\ast\cdots  \ \textup{where}  \ \nu_n = \frac12(\delta_0+ \delta_{2\cdot4^{-n}})
\end{equation}
is a spectral measure}  with a spectrum
\[
\Lambda=\displaystyle\left\{\sum_{k=0}^{n}4^{k}\ell_{k}:\ell_{k}\in\{0,1\},n\geq0\right\}.
\]
Jorgensen and Pedersen's seminal work opened up a new field in researching the orthogonal harmonic
analysis of fractal measures. Since then, singularly continuous spectral measures have been extensively studied, see \cite{DHL} and the references therein.

\subsection{Product Spectral set problem.} Product spectral set problem can be stated as follows:

\begin{problem}\label{conjecture}
Let $\Omega_1\subset\mathbb{R}^{n}$ and $\Omega_2\subset\mathbb{R}^{m}$ be two bounded, measurable sets of positive Lebesgue measure. Is it true that  $\Omega=\Omega_1\times \Omega_2$ is a spectral set if and only if $\Omega_1$ and $\Omega_2$ are both spectral sets? 
\end{problem}

The problem was first studied by Greenfeld and Lev \cite{GL-JFA} who showed that the problem is true if $\Omega_1 = [0,1]^n$. In \cite{KMN-BHM}, by a different method, Kolountzakis showed that the problem also holds if $\Omega_1\subset\mathbb{R}$ is a union of two intervals. Moreover,  the problem is also true when $\Omega_1$ is a convex polygon in two dimensions by Greenfeld and Lev \cite{GL-JAM}. It was shown to be true if $\Omega_1$ is a convex set \cite{KLM-Signal} after the breakthrough result of Lev and Matolcsi \cite{LM-AM}.  Recently, Chen, Liu and Zheng \cite{CLZ-JGA} proved that the problem is true if
$\Omega_1$ is the classical Sierpinski self-affine tile and the Lebesgue measure of $\Omega_2$ is equal to 1. Ramabadran and Van Vilet \cite{RV-ARX} showed that the problem is true if $\Omega_1$ is a measurable set of measure 1 supported inside $[0,3/2-\varepsilon]$ for some $\varepsilon>0$. 

Using the non-spectral tile in \cite{KM-FM}, Somlai \cite{S-ARX} constructed two bounded measurable sets $\Omega_1, \Omega_2\subset \R^3$  such that $\Omega_1$ is not spectral, but the product set $\Omega=\Omega_1\times \Omega_2$ is a spectral set in $\mathbb{R}^{6}$. Therefore the problem is not true for $n = m \ge 3.$  However, this problem is still open when one of the factors lies in dimension 1 or 2. Indeed, if the problem is false, then one side of the Fuglede's conjecture will be false as well (see \cite{KMN-BHM}). The problem is also unknown if we assume one of the factors is spectral in $\R^1$. We are going to make some new progress under this assumption. 

\subsection{Main Results.} Product spectral set problem can be naturally generalized to product spectral measures. It is direct to prove that if $\mu,\nu$ are spectral measures with their respective spectra $\Lambda_1,\Lambda_2$, then $\Lambda_1\times\Lambda_2$ is a spectrum for $\rho  = \mu\times \nu$, see, for example, \cite[Proposition 2.2]{LW-ACHA}. 
The converse is the only difficult part. Denote the zero set of the Fourier transform of $\mu$ by ${\mathcal Z}(\widehat{\mu})$. 
$$
\mathcal{Z}(\widehat{\mu}) = \{\xi: \widehat{\mu}(\xi) = 0\}.
$$
Our general main result, that will lead to other special cases, is presented below:

\begin{theorem}\label{spectral}
Let $\mu,\nu$ be two finite Borel  measures defined respectively on $\R^{n}$ and $\R^{m}$ and let $\Gamma$ be an {additive} subgroup of $\R^{n}$.  Let $\rho=\mu\times \nu$ be its product measure defined on $\R^{n+m}$. Suppose that  $\mu$ satisfies
\begin{equation}\label{ldcfb}
\mathcal{Z}(\widehat{\mu})\cup\{0\}  \subset  \Gamma.
\end{equation}
Then  \begin{enumerate}
\item if $\rho$ is a spectral measure, then $\rho$ admits a spectrum $\Lambda$
such that
$$
\Lambda \subset \Gamma \times \R^m.
$$
\item If (\ref{ldcfb}) is an equality, then $\nu$ is a spectral measure and $\mu$ is a spectral measure with a spectrum  $\Gamma$.
\end{enumerate}
\end{theorem}

The proof of this theorem is a refinement of the proof by Greenfeld and Lev \cite{GL-JFA} and it generalizes their main theorem in two ways. First, none of the measures are assumed to be Lebesgue measure supported on a set and they can even be singular measures of fractal type. Second, (\ref{ldcfb}) is assumed only to be contained in an additive subgroup and this subgroup does not even need to have any topological structure, while their original proof was only concerned with $\mu$ being the Lebesgue measure on an interval which means that (\ref{ldcfb}) is an equality.  One of the key ingredients for the proof in \cite{GL-JFA} used the fact that weak limit of spectra of a measurable set of finite positive measure preserves to be a spectrum. We will see that this is not the case for singular measures in Section \ref{3.1}. We will bypass the weak convergence problem by carefully exploiting the structure of the spectrum under piecewise local translations. 

Let us now discuss some immediate corollaries of Theorem \ref{spectral} (2). First,   by an induction on $k$, we have  

\begin{corollary}\label{cor-spectral}
    Let $\rho = \mu\times \nu$ where $\mu = \mu_1\times\cdots\times\mu_k$ and for each $i=1,\cdots,k$, ${\mathcal Z}(\widehat{\mu_i})\cup\{0\}$ is an {additive} subgroup in $\R^{d_i}$. Then $\rho$ is spectral in $\R^{d_1+\cdots+d_k+d}$ implies that $\nu$ is also spectral in $\R^{d}$. 
\end{corollary}

 Denote by ${\mathcal L}|_{\Omega}$ the  Lebesgue measure on a measurable set $\Omega$. Since ${\mathcal Z}(\widehat{\mathcal{L}|_{[0,1]}}) = \Z\setminus\{0\}$, we have 
\begin{corollary}\label{interval}
Let $\nu$ be a {finite} Borel measure on $\R^m$ with compact support. Then  $\mathcal{L}|_{[0,1]^d}\times \nu$ is a spectral measure on $\R^{d+m}$ if and only if $\nu$ is a spectral measure on $\R^m.$
\end{corollary}
This generalizes the result of Greenfeld and Lev from spectral sets to spectral measures in the second component.  Moreover, this also immediately recovers the result by Chen, Liu and Zheng \cite{CLZ-JGA} since the zero of the Fourier transform of the classical Sierpinski self-affine tiles is equal to $\Z^2\setminus\{0\}$, which was computed on \cite[p.26]{CLZ-JGA}. 

\medskip

We say that $(A,B)$ is a {\bf complementary pair} for $N$ if 
$$
A\oplus B = \{0,1,\cdots, N-1\}
$$
under the addition of integers and every integer in $\{0,1,\cdots, N-1\}$ can be written uniquely as $a+b$ for $a\in A$ and $b\in B$. The structure of $A$ and $B$ has been completely classified by C. Long in 1967 \cite{Long} . It is also well-known that $A+[0,1]$ is a spectral set (see also \cite{GL-2014,PW}).  Using  Theorem \ref{spectral}, we will prove the following theorem for the product spectral measures. 
\begin{theorem}\label{product-tiling interval-intro}
    Let $\mu$ be the Lebesgue measure for the set $A+[0,1]$ where $A$ forms a complementary pair for $N$ with some subset $B$. Suppose that $\mu\times \nu$ is spectral {in $\R^{1+m}$}. Then $\nu$ is spectral in $\R^{m}$.
\end{theorem}

A simple example of complementary pair  is $A = \{0,2\}$ and $B = \{0,1\}$, so product spectral set problem (Problem \ref{conjecture}) is true if one of the factors is the Lebesgue measure supported on $[0,1]\cup [2,3]$. The above theorem asserts a wide class of finite unions of intervals also satisfies the product spectral set problem. This is new even in the ordinary product spectral set problem.

\subsection{Notations and organization.} In this paper, we will  write 
$$
e_{\lambda}(x) = e^{2\pi i \lambda\cdot x}
$$
{For a finite positive Borel measure $\tau $, we denote by $\Norm{\tau }:=\tau(\R^n)$ its total measure in $\R^n$. Then the exponential functions $\{e_{\lambda}:\lambda\in \Lambda\}$ is an orthogonal basis for $L^2(\tau)$ if and only if it is mutually orthogonal and the Parseval identity of the following form holds:}
$$
\sum_{\lambda\in\Lambda}|\langle f,e_{\lambda}\rangle_{L^2(\tau)}|^2=\Norm{\tau}\cdot\|f\|_{L^2(\tau)}^2.
$$
A basic criterion for determining whether $\Lambda\subset\mathbb{R}^{d}$ is a spectrum for $\mu$ was given by Jorgensen and Pedersen \cite{dense} and will be used throughout the paper. In the original version, $\mu$ was a probability measure,  but it is direct to see that it holds in the following form by normalizing the total mass. 
\begin{theorem}[\cite{dense}]\label{Q}
Let $\mu$ be a finite Borel measure with compact support in $\mathbb{R}^{d}$, $\Lambda$ be a countable subset of $\mathbb{R}^{d}$. Let
\[
Q(\xi):=\sum_{\lambda\in \Lambda}|\hat{\mu}(\xi+\lambda)|^{2}.
\]
 Then
\begin{enumerate}
  \item [\textup{(1)}] $E(\Lambda)$ is an orthogonal set of $L^{2}(\mu)$ if and only if $Q(\xi)\leq \Norm{\mu}^2$ for $\xi \in \mathbb{R}^{d}$. In this case, $Q(\xi)$ has an entire analytic extension to $\C^d$.
\item [\textup{(2)}] $E(\Lambda)$ is an orthogonal basis for $L^{2}(\mu)$ if and only if $Q(\xi)\equiv \Norm{\mu}^2$ for $\xi \in \mathbb{R}^{d}$.
\end{enumerate}
\end{theorem}

We will prove Theorem \ref{spectral} in Section \ref{sec2} and Theorem \ref{product-tiling interval-intro} in Section \ref{sec3}. Finally, we will give some examples and the examples that fail the weak convergence in Section \ref{sec4}.

\section{The Proof of Theorem \ref{spectral}}\label{sec2}

\subsection{Setup of the proof. } Throughout the section, we let
\begin{equation}\label{measure}
\rho = \mu\times \nu    
\end{equation}
be a finite Borel measure in $\R^{n+m}$ where $\mu$ and $\nu$ are respectively the finite Borel measures in  and $\R^n$ and $\R^m$. Moreover, ${\mathcal Z}(\widehat{\mu})\cup\{0\}$ is contained in an {additive} subgroup of $\R^n$, which we will denote by $\Gamma$.  

Assume that $\Lambda_{1}$ and $\Lambda_{2}$ are two discrete sets in $\mathbb{R}^{n+m}$  and $\Lambda_{2}$, $\Lambda_{2}+\gamma$ and $\Lambda_{2}-\gamma$ are disjoint from $\Lambda_{1}$, where $\gamma$ is a vector in $\mathbb{R}^{n+m}$.
Define
\begin{equation}\label{partition}
\Lambda:=\Lambda_{1}\cup \Lambda_{2},\ \ \ \Lambda^{+}_\gamma:=\Lambda_{1}\cup (\Lambda_{2}+\gamma),\ \ \ \Lambda^{-}_\gamma:=\Lambda_{1}\cup (\Lambda_{2}-\gamma).
\end{equation}
We begin by recalling a simple lemma  in \cite[Lemma 4.1]{GL-JFA}. 

\begin{lemma}\label{3or}Let $\Lambda$, $\Lambda^{+}_\gamma$ and $\Lambda^{-}_\gamma$ be the sets defined in $\eqref{partition}$, $\rho=\mu\times\nu$ be the measure defined in $\eqref{measure}$.
Suppose that $E(\Lambda)$ is an orthogonal basis for $L^2(\rho)$, and both $E(\Lambda^{+}_\gamma)$ and $E(\Lambda^{-}_\gamma)$ are orthogonal sets of $L^2(\rho)$, then each one of the sets $E(\Lambda^{+}_\gamma)$ and $E(\Lambda^{-}_\gamma)$ is an orthogonal basis for $L^2(\rho)$.
\end{lemma}

For a vector $\xi$ in $\mathbb{R}^{n+m}$, we write $\xi=(\xi_{1},\xi_{2})$, where $\xi_{1}\in\mathbb{R}^n$ and $\xi_{2}\in\mathbb{R}^m$. It is direct to see that
\[
\widehat{\rho}(\xi)=\widehat{\mu}(\xi_{1})\cdot \widehat{\nu}(\xi_{2}).
\]
This implies  that
\[
\mathcal{Z}(\widehat{\rho})=\{\xi=(\xi_{1},\xi_{2})\in\mathbb{R}^{n+m}:\text{$\xi_{1}\in\mathcal{Z}(\widehat{\mu})$ ~or ~$\xi_{2}\in\mathcal{Z}(\widehat{\nu})$}\}.
\]
The following lemma is immediate from the definition of orthogonal sets.

\begin{lemma}\label{zero}
Let $\lambda =(\lambda_1,\lambda_2)$ and $\lambda'=(\lambda'_1,\lambda_2')$ be two distinct points in $\mathbb{R}^{n+m}$, and $\rho=\mu\times\nu$ be the measure defined in \eqref{measure}. The exponential functions $e_{\lambda}$ and $e_{\lambda'}$ are orthogonal in $L^{2}(\rho)$ if and only if $\lambda_{1}-\lambda'_{1}\in \mathcal{Z}(\widehat{\mu})$, or $\lambda_{2}-\lambda'_{2}\in \mathcal{Z}(\widehat{\nu})$. 
\end{lemma}


Let $ \Lambda\subset \R^{n+m}$ be a spectrum for $\rho$. Let us write $\lambda = (\lambda_1,\lambda_2)\in \R^{n+m}$. Given $t\in \R^n$ we define a mapping  $\phi_{t}$ on $\Lambda$ by
\begin{equation}\label{mapping}
\phi_{t}(\lambda):=
\left\{\begin{matrix}
\lambda,\ \ \ \ \ \ \ \ \ & \ \lambda_{1}\in  \Gamma,\\
\lambda+\tau(t),& \lambda_{1}\notin \Gamma,
\end{matrix}\right.
\end{equation}
where $\tau(t):=(t,0)\in \mathbb{R}^{n+m}$. 

\begin{proposition}\label{mapspectrum}
 Given $t\in \R^n$. Let $\phi_{t}$ be the mapping defined by $\eqref{mapping}$, $\rho=\mu\times\nu$ be the measure defined in $\eqref{measure}$, and ${\mathcal Z}(\widehat{\mu})\cup \{0\}$ is a subset of  the {additive} subgroup $\Gamma$. If $\Lambda$ is a spectrum of $\rho$, then
$\phi_{t}(\Lambda)$ is also a spectrum of $\rho$.
\end{proposition}
\begin{proof}

We first prove that $E(\phi_t(\Lambda))$ is mutually orthogonal in $L^2(\rho)$. Given distinct $\lambda = (\lambda_1,\lambda_2)$ and $\lambda' = (\lambda_1',\lambda_2')$ in $\Lambda$. We have three cases to consider according to the definition of $\phi_t$. 
\begin{enumerate}
    \item  $\lambda_1,\lambda_1'\in\Gamma$;
    \item  $\lambda_1,\lambda_1'\not\in\Gamma$;
    \item  $\lambda_1\in\Gamma$ and $\lambda_1'\not\in \Gamma$ or vice versa. 
\end{enumerate}

In the first two cases, $\phi_t(\lambda)-\phi_t(\lambda') = \lambda-\lambda'$. As $e_{\lambda}$, $e_{\lambda'}$ are mutually orthogonal in $L^2(\rho)$, we know the same is also true for $e_{\phi_t(\lambda)}$, $e_{\phi_t(\lambda')}$.

In the last case, we only need to consider one situation since the role of $\lambda_1,\lambda_1'$ is symmetric.  Now,
\begin{equation}\label{phi_t-difference}
\phi_t(\lambda)-\phi_t(\lambda') = (\lambda_1-\lambda_1'-t,\lambda_2-\lambda_2').
\end{equation}
From our assumption, $\lambda_1\in \Gamma$, while $\lambda_1'\not\in \Gamma$. As  $\Gamma$ is an {additive} group, we have $\lambda_1-\lambda_1'\not\in\Gamma$. By Lemma \ref{zero} and ${\mathcal Z}(\widehat{\mu})\subset \Gamma$, it means that $\lambda_1-\lambda_1'\not\in{\mathcal Z}(\widehat{\mu})$ and thus $\lambda_2-\lambda_2'\in{\mathcal Z}(\widehat{\nu})$. Hence, $\phi_t(\lambda)-\phi_t(\lambda')\in {\mathcal Z}(\widehat{\rho})$ in view of the second component in (\ref{phi_t-difference}). This completes the proof of the mutual orthogonality. A similar argument also  shows that  $E(\phi_{-t}(\Lambda))$ is also a mutually orthogonal set.

 We next show that $\phi_{t}(\Lambda)$ is indeed a spectrum of $\rho$. To this end, we  divide the set $\Lambda$ into two subsets, $\Lambda_{1}$ and $\Lambda_{2}$, such that they satisfy $\eqref{partition}$ and the assumptions in Lemma \ref{3or}.

Let
\[
\Lambda_{1}:=\{\lambda\in \Lambda:\lambda_{1}\in \Gamma\},\ \ \Lambda_{2}:=\{\lambda\in \Lambda:\lambda_{1}\notin \Gamma\}.
\]
We have immediately that 
$$
\phi_{t}(\Lambda)=\Lambda_{1}\cup(\Lambda_{2}+\tau(t)), \ \ \phi_{-t}(\Lambda)=\Lambda_{1}\cup(\Lambda_{2}-\tau(t)).$$
Moreover,
$$\Lambda_{1}\cap(\Lambda_{2}+\tau(t))=\varnothing, \ \ \Lambda_{1}\cap (\Lambda_{2}-\tau(t))=\varnothing.$$
Indeed, if either one of the the intersections is not disjoint, we will find some $(\lambda_1,\lambda_2) = (\lambda_1'+t,\lambda'_2)$ or $(\lambda_1'-t,\lambda'_2)$ where $(\lambda_1',\lambda_2')\in\Lambda_2$. In both cases, it implies that $\lambda_2 = \lambda_2'$. Since $(\lambda_1-\lambda_1',\lambda_2-\lambda_2')\in{\mathcal Z}(\widehat{\rho})$,   $\lambda_1-\lambda_1'\in{\mathcal Z}(\widehat{\mu})\subset \Gamma$ by the mutually orthogonality of $E(\Lambda)$ and Lemma \ref{zero}. As $\lambda_1\in\Gamma$, this forces $\lambda_1'\in\Gamma$ which is a contradiction since $\lambda_1'\in\Lambda_2$.

Finally, we can invoke  Lemma \ref{3or} to conclude that {each of} $\phi_{t}(\Lambda)$ and $\phi_{-t}(\Lambda)$ is a spectrum of $\rho$.
\end{proof}

The following lemma shows that if we can construct a spectrum $\Lambda$ such that there is only one equivalence class, then $\nu$ must be spectral.

\begin{lemma}\label{spectrum}
Assume that $\rho=\mu\times\nu$ is the product measure defined in $\eqref{measure}$, and $\Upsilon$ is a mutually orthogonal set for $\mu$ (here $\Upsilon$ is not necessarily an additive subgroup). 
If $\rho$ admits a spectrum 
$$
\Lambda\subset \Upsilon\times \mathbb{R}^{m},
$$ then each set
\[
\Lambda_{\gamma}:=\{\lambda\in\mathbb{R}^{m}:(\gamma,\lambda)\in \Lambda\}, ~~\gamma\in \Upsilon,
\]
is a spectrum of $\nu$ and $\mu$ is a spectral measure with spectrum $\Upsilon$. 
\end{lemma}
\begin{proof}
 Write $\Lambda = \bigcup_{\gamma\in\Upsilon}\{\gamma\}\times \Lambda_{\gamma}$. 
 Moreover, each $\Lambda_{\gamma}$ must be mutually orthogonal for $\nu$ and $\Upsilon$ is also mutually orthogonal for $\mu$. Therefore,
    \[
\begin{split}
\Norm{\rho}^2&=\sum_{(\gamma,\lambda)\in \Lambda}|\widehat{\rho}((\xi_1-\gamma,\xi_{2}-\lambda))|^2\nonumber\\
&=\sum_{\gamma\in\Upsilon}\sum_{\lambda\in\Lambda_{\gamma}}|\widehat{\mu}(\xi_1-\gamma)|^2\cdot
|\widehat{\nu}(\xi_{2}-\lambda)|^2\nonumber\\
&\le \Norm{\nu}^2\sum_{\gamma\in\Upsilon}|\widehat{\mu}(\xi_1-\gamma)|^2\le \Norm{\mu}^2\Norm{\nu}^2=\Norm{\rho}^2.
\end{split}
\]
This forces that all inequalities hold as equalities. In particular, all $\Lambda_{\gamma}$ are spectra for $\nu$ and $\Upsilon$ is a spectrum for $\mu$. 
\end{proof}

\subsection{Completion of the proof.} Suppose $\Lambda$ is a spectrum for $\rho$ and $\Gamma$ is an additive group in $\R^n$. We define an equivalence relation on $\Lambda$ by 
$$
\lambda\sim\lambda'~ \Longleftrightarrow ~\pi_1(\lambda)-\pi_1(\lambda')\in \Gamma, 
$$
where $\pi_1$ is the orthogonal projection to the first $n$ coordinates. 
Let $\{[\lambda_n]: n\in\N\}$ be the collection of all the equivalence classes under this equivalence relation where $\lambda_n\in\Lambda$ for all $n\in\N$. We can choose $\lambda_1 = 0$ since $0\in\Lambda$. From here, for each $\gamma\in\Gamma$, we let 
$$
\Lambda_{n,\gamma} = \{y_2: \lambda=(\pi_1(\lambda_n)+\gamma,y_2)\in \Lambda\}.
$$
Then $\Lambda$ admits a natural decomposition according to the equivalence class. 
\begin{equation}\label{eq-equivalence class}
\Lambda = \bigcup_{n=1}^{\infty}\Lambda_n, \ \mbox{where}~ \Lambda_n =  \bigcup_{\gamma\in\Gamma} \{x_n+\gamma\}\times \Lambda_{n,\gamma}
\end{equation}
where we let $x_n = \pi_1(\lambda_n)$ and $x_1 = 0$. We can now  prove Theorem \ref{spectral}.
\smallskip

\begin{proof}[Proof of Theorem \ref{spectral} (1)]
We first take  a spectrum $\Lambda$ for $\rho$ and we decompose $\Lambda$ according to its equivalence class as in (\ref{eq-equivalence class}):
\begin{equation}\label{eq-dec}
\Lambda = \bigcup_{n=1}^{\infty} \left( \bigcup_{\gamma\in\Gamma} \{x_n+\gamma\}\times \Lambda_{n,\gamma}\right).
\end{equation}

We have the following simple observations:
$$
\phi_{-x_2}(\Lambda) =  \left(\bigcup_{n=1}^2\bigcup_{\gamma\in\Gamma} \{\gamma\}\times \Lambda_{n,\gamma} \right)~\cup~ \left(\bigcup_{n=3}^{\infty}\bigcup_{\gamma\in\Gamma} \{x_n-x_2+\gamma\}\times \Lambda_{n,\gamma}\right)
$$
$$
\phi_{-(x_3-x_2)}\circ\phi_{-x_2}(\Lambda)=\left( \bigcup_{n=1}^3\bigcup_{\gamma\in\Gamma} \{\gamma\}\times \Lambda_{n,\gamma}\right) ~\cup~\left( \bigcup_{n=4}^{\infty}\bigcup_{\gamma\in\Gamma} \{x_n-x_3+\gamma\}\times \Lambda_{n,\gamma}\right)
$$
Inductively, if we define $\Lambda^{(m)} = \phi_{-(x_{m}-x_{m-1})}\circ\cdots\circ\phi_{-(x_3-x_{2})}\circ\phi_{-x_2}(\Lambda)$, we have 
$$
\Lambda^{(m)}= \left(\bigcup_{n=1}^m\bigcup_{\gamma\in\Gamma} \{\gamma\}\times \Lambda_{n,\gamma}\right) ~\cup~ \left(\bigcup_{n=m+1}^{\infty}\bigcup_{\gamma\in\Gamma} \{x_n-x_m+\gamma\}\times \Lambda_{n,\gamma}\right)
$$
All these are spectra for $\rho$ by Proposition \ref{mapspectrum}. Finally, we apply $\phi_{x_m}$ to $\Lambda^{(m)}$ and obtain 
$$
\begin{aligned}
\Lambda_m' = &\left(\bigcup_{n=1}^m\bigcup_{\gamma\in\Gamma} \{\gamma\}\times \Lambda_{n,\gamma} \right)~\cup~ \left(\bigcup_{n=m+1}^{\infty}\bigcup_{\gamma\in\Gamma} \{x_n+\gamma\}\times \Lambda_{n,\gamma}\right)\\
: = & A_m\cup B_m
\end{aligned}
$$
 where
$$
A_m = \bigcup_{n=1}^m\bigcup_{\gamma\in\Gamma} \{\gamma\}\times \Lambda_{n,\gamma}, \ B_m = \bigcup_{n=m+1}^{\infty}\bigcup_{\gamma\in\Gamma} \{x_n+\gamma\}\times \Lambda_{n,\gamma}.
$$
For each $m\ge 1$, $\Lambda_m'$ is also a spectrum for $\rho$. Note that $B_m$ is a part of the orthogonal set  in $\Lambda$. In view of (\ref{eq-dec}) and $|\widehat{\rho}|^2$ is non-negative, we have 
\begin{equation}\label{eq-B_m-to 0}
\sum_{\lambda\in B_m} |\widehat{\rho}(\xi-\lambda)|^2 \to 0, \ \mbox{as} \ m\to\infty.
\end{equation}
Finally, we let 
$$
\Lambda' = \bigcup_{m=1}^{\infty} A_m \subset \Gamma\times \R^m.
$$
As $\Lambda_m'$ is a spectrum for $\rho$, we have 
$$
\Norm{\rho}^2= \sum_{\lambda\in\Lambda_{m}'} |\widehat{\rho}(\xi-\lambda)|^2 = \sum_{\lambda\in A_m} |\widehat{\rho}(\xi-\lambda)|^2+ \sum_{\lambda\in B_m} |\widehat{\rho}(\xi-\lambda)|^2. 
$$ 
Using (\ref{eq-B_m-to 0}), we see that 
$$
\begin{aligned}
 \sum_{\lambda\in\Lambda'}|\widehat{\rho}(\xi-\lambda)|^2 =&\lim_{m\to\infty} \sum_{\lambda\in A_m} |\widehat{\rho}(\xi-\lambda)|^2\\
 =&\lim_{m\to\infty}\left(\Norm{\rho}^2-\sum_{\lambda\in B_m} |\widehat{\rho}(\xi-\lambda)|^2\right)=\Norm{\rho}^2.
\end{aligned}
$$
This shows that $\Lambda'$ is our desired spectrum for $\rho$ and completes the proof. 
\end{proof}

\begin{proof}[Proof of Theorem \ref{spectral} (2).] Suppose that ${\mathcal Z}(\widehat{\mu})\cup\{0\} = \Gamma$ and $\Gamma$ is an {additive} subgroup of $\R^n$. Then $\Gamma$ is a mutually orthogonal set for $\mu$. By Theorem \ref{spectral} (1), we can always find a spectrum $\Lambda'$ such that 
$$
\Lambda'\subset \Gamma \times \R^m.
$$
But $\Gamma$ is mutually orthogonal for $\mu$. The desired conclusion follows from Lemma \ref{spectrum}.
\end{proof}
\medskip

\section{The proof of Theorem \ref{product-tiling interval-intro}}\label{sec3}

To prove Theorem \ref{product-tiling interval-intro}, let $\N_n = \{0,1,\cdots, n-1\}$ for convenience and recall the classical result of C. Long \cite{Long} (see also \cite[Corollary 2.6]{GL-2014}). Indeed,  $A,B$ is a complementary pair for $N$ if  and only if  after possibly interchanging the role of  $A$ and $B$, we can decompose $N = n_0\cdots n_{k-1}$ for some integers $n_j>1$, and writing ${\bf m}_j = n_0\cdots n_j$, one has
\begin{align}
   A &= n_0\N_{n_1}+ {\bf m}_2\N_{n_3}+\ldots+{\bf m}_{k-2}\N_{n_{k-1}}\nonumber\\
  B &= \N_{n_0}+{\bf m}_1\N_{n_2}+\ldots+{\bf m}_{k-3}\N_{n_{{k-2}}}.\nonumber 
\end{align}
  As a simple example, readers can consider $A = \{0,2\} = 2\N_2$ and $B = \{0, 1\} = \N_2$ for the proof. In this case, $A\oplus B = \{0,1,2,3\}$ and $(A,B)$ is a complementary pair for $4$.

Let 
$$
{\mathbb O}_n = \Z\setminus n\Z
$$
which is the set of integers that are not multiples of $n$. Let us first compute the zeros of the Fourier transform of $\mu$. Indeed, 
$$
\widehat{\delta_{\N_{n}}}(\xi) = \frac1n \sum_{j=0}^{n-1}e^{-2\pi i j\xi} = \frac{e^{-2\pi i n\xi}-1}{n(e^{-2\pi i \xi}-1)}.
$$
This shows that 
$${\mathcal Z}(\widehat{\delta_{\N_{n}}}) = \frac1n{\mathbb O}_n.
$$
In particular, $${\mathcal Z}(\widehat{\delta_A})=\frac{1}{{\bf m}_1}{\mathbb O}_{n_1}\cup\frac{1}{{\bf m}_3}{\mathbb O}_{n_3}\cup\cdots\cup\frac1{{\bf m}_{k-1}}{\mathbb O}_{n_{k-1}}$$ and
$$
{\mathcal Z}(\widehat{\delta_B})=\frac{1}{{\bf m}_0}{\mathbb O}_{n_0}\cup\frac{1}{{\bf m}_2}{\mathbb O}_{n_2}\cup\cdots\cup\frac1{{\bf m}_{k-2}}{\mathbb O}_{n_{k-2}}
$$
(notice that ${\mathbf m}_{k-1} = N$). Observe that
$$
{\mathbb O}_{n_{j-1}} \cup n_{j-1} {\mathbb O}_{n_{j-2}} = {\mathbb O}_{n_{j-1}n_{j-2}}.
$$
Combining the union into the same denominators, and iterating the above relationship, we see that 
\begin{equation}\label{eq-partition1}
\begin{aligned}
     {\mathcal Z}(\widehat{\delta_A})\cup{\mathcal Z}(\widehat{\delta_B}) = &\frac{1}{N} \left({\mathbb O}_{n_{k-1}}\cup n_{k-1}{\mathbb O}_{n_{k-2}}\cup n_{k-1}n_{k-2}{\mathbb O}_{n_{k-3}}\cup\cdots\cup n_1\cdots n_{k-1}{\mathbb O}_{n_0}\right)\\
     =& \frac{1}{N}{\mathbb O}_{N}.
\end{aligned}
\end{equation}   
Moreover, using the uniqueness of the representation of the integers in $\Set{0,\ldots,N-1}$ using the basis $n_0,\cdots, n_{k-1}$ (i.e. each $n \in{\mathbb O}_N$ can be uniquely represented as $$\varepsilon_0+\varepsilon_1 n_0+\cdots \varepsilon_{k-1}n_0\cdots n_{k-2}$$ 
for unique $\varepsilon_j\in {\mathbb N}_{n_j}$), we can show that
\begin{equation}\label{eq-partition2}{\mathcal Z}(\widehat{\delta_A})\cap{\mathcal Z}(\widehat{\delta_B})=\varnothing.\end{equation}
(\ref{eq-partition1}) and (\ref{eq-partition2}) leads us to conclude the following lemma:

\begin{lemma}\label{partitiion-lemma}
    Suppose that $(A,B)$ forms a complementary pair for $N$. Then the zero sets of $\widehat{\delta_A}$ and $\widehat{\delta_B}$ form a partition of $\frac{1}{N}{\mathbb O}_{N}$.
\end{lemma}

{Furthermore, (see e.g. \cite{GL-2014}), the measure $\delta_A$ and $\delta_B$ are spectral measures with spectrum}
\begin{equation}\label{LambdaA}
{\Lambda_A=\frac{1}{{\bf m}_1}\N_{n_1}+\frac{1}{{\bf m}_3}\N_{n_3}+\ldots+\frac{1}{{\bf m}_{k-1}}\N_{n_{k-1}}}
\end{equation}
and 
\begin{equation}\label{LambdaB}
{\Lambda_B=\frac{1}{{\bf m}_0}\N_{n_0}+\frac{1}{{\bf m}_2}\N_{n_2}+\ldots+\frac{1}{{\bf m}_{k-2}}\N_{n_{k-2}},}
\end{equation}
respectively.
We first prove a general lemma concerning completeness of exponentials in direct sums. 

{For a finite set $D\subset \mathbb{R}^n$, we denote by
\[
{\delta}_D:=\sum_{d\in D}\delta_d
\]
the  discrete measure associated with $D$.  We first have the following lemma about the completeness of the exponentials when combining different translates of the spectrum for convolutions of measures. 

\begin{lemma}\label{lm:conv}
Suppose that
\begin{enumerate}
\item $\varrho$ is a finite nonnegative measure on $\RR^n$ with the frequencies $\Lambda \subseteq \RR^n$ being a spectrum for $\varrho$,
\item $D \subseteq \RR^n$ is a finite set such that $(\supp\varrho) + D$ is a packing (i.e. the intersections of the support of any two different measures $\varrho_1=\delta_{d_1}\ast\varrho$ and $\varrho_2 = \delta_{d_2}\ast\varrho$ have measure zero in each $\varrho_i$).
\item $T$ is a finite set with $\Abs{T} = \Abs{D}$ such that $\Lambda+T$ is a direct sum (all sums are different) and is an orthogonal set for $L^2(\varrho*{\delta}_D)$.
\end{enumerate}
Then $\Lambda+T$ is also complete for $L^2(\varrho*\widetilde{\delta}_D)$.
\end{lemma}


\begin{proof}
Let $\{ e_{n}:n\in\N\}=\{e_{\lambda+t}:\lambda\in \Lambda,t\in T\}$. Since $\{ e_n\}$ is orthogonal we can complete it to an orthogonal basis $\mathcal{D}$ adding the vectors $\{ \upsilon_m:m\in \N\}$ with $\Norm{\upsilon_m}=\Norm{e_n}$.
Let $f\in L^2 (\varrho\ast{\delta}_D)$, we can write $f = \sum_{d\in D} f_d$ for some unique $f_d \in L^2(\varrho*{\delta}_D)$ which is non-zero only on the set $(\supp\varrho)+d$, for some $d\in D$. For such an $f_d$ we have

\begin{align}\label{fd}
\Norm{f_d}^2_{L^2(\varrho*{\delta}_D)}
&=\Norm{f_d}^2_{L^2(\varrho*\delta_d)}
\end{align}
{since $f_d=0$ off $(\supp\varrho)+d$ and the assumption $(2)$}. Moreover, we have
\begin{align*}
\sum_{\lambda \in \Lambda, t \in T} \Abs{\inner{f_d}{e_{\lambda+t}}_{L^2(\varrho*{\delta}_D)}}^2
 &= \sum_{t \in T} \sum_{\lambda \in \Lambda} \Abs{\inner{f_d}{e_{\lambda+t}}_{L^2(\varrho*\delta_d)}}^2\ \ \ \text{ (since $f_d=0$ off $(\supp\varrho)+d$)}\\
 &= \sum_{t \in T} \Norm{\varrho} \cdot \Norm{f_d}^2_{L^2(\varrho*\delta_d)}\ \ \ \text{ (since $\Lambda+t$ a spectrum for $\varrho*\delta_d$)}\\
 &= \Abs{D} \cdot \Norm{\varrho} \cdot \Norm{f_d}^2_{L^2(\varrho*\delta_d)}\ \ \ \text{(since $\Abs{T}=\Abs{D}$)}\\
 &= \Abs{D} \cdot \Norm{\varrho} \cdot \Norm{f_d}^2_{L^2(\varrho*{\delta}_D)}\ \ \ (\text{since (\ref{fd})})\\
 &= \Norm{\varrho*{\delta}_D} \cdot \Norm{f_d}^2_{L^2(\varrho*{\delta}_D)}.
\end{align*}
But we also have 
\[\Norm{\varrho*{\delta}_D}\cdot \Norm{f_d}^2_{L^2(\varrho*{\delta}_D)}=\sum_{n} \Abs{\inner{f_d}{e_{n}}_{L^2(\varrho*{\delta}_D)}}^2+\sum_{m} \Abs{\inner{f_d}{\upsilon_{m}}_{L^2(\varrho*{\delta}_D)}}^2\]
using Parseval with respect to $\mathcal{D}$. This implies that $\inner{f_d}{\upsilon_{m}}_{L^2(\varrho*{\delta}_D)}=0$ for all $m$ and we can write $f_d=\sum_{n}{\frac{1}{\Norm{\varrho*{\delta}_D}}}\inner{f_d}{e_{n}}_{L^2(\varrho*{\delta}_D)}e_n$. Hence, 
\[f=\sum_{d\in D} f_d=\sum_{n}{\frac{1}{\Norm{\varrho*{\delta}_D}}}\left(\sum_{d\in D}\inner{f_d}{e_{n}}_{L^2(\varrho*{\delta}_D)}\right)e_n=\sum_{n}{\frac{1}{\Norm{\varrho*{\delta}_D}}}\inner{f}{e_{n}}_{L^2(\varrho*{\delta}_D)}e_n,\]
which implies that $\Lambda+T$ is complete for $L^2(\varrho*{\delta}_D)$.
\end{proof}

\begin{proof}[Proof of Theorem \ref{product-tiling interval-intro}.] Recall that $\mu$ is the Lebesgue measure supported on $A+[0,1]$ with $(A,B)$ forming a complementary pair for $N$. 
 $$\eta := \mathds{1}_{[0, N]}\times \nu $$

\medskip

{\bf Claim:} $\eta$ is a spectral measure. 

\medskip

\noindent If the claim holds. Since $\mu = \delta_A\ast\mathds{1}_{[0, 1]}$, so we have $\mathds{1}_{[0, N]} = \mu*\delta_B$  and
$$
\Set{0} \cup \mathcal{Z}(\ft\mu) \cup \mathcal{Z}(\ft{\delta_B}) = \Set{0} \cup \mathcal{Z}(\ft{\mathds{1}_{[0, N]}}) = \frac1N\ZZ,
$$
{which is an additive subgroup of $\RR$.
Applying Theorem \ref{spectral} (2), $\nu$ is also spectral, thereby completing the proof.}

We now justify the claim. First we notice that we can rewrite $\eta$ as follows:
$$
\eta = \delta_{B\times\Set{0}}*(\mu\times\nu).
$$
By (\ref{eq-partition1}), ${\mathcal Z}(\widehat{\delta}_A) \subset \frac{1}{N}\ZZ$. As $\mu\times \nu$ is spectral,  Theorem \ref{spectral} (1) shows that   there is a spectrum $\Lambda$ for 
$\mu\times\nu$ such that  
\begin{equation}\label{eq-Lambda-1.5}
\Lambda\subset (\frac1N\ZZ) \times \RR^m.
\end{equation}
Recall that $\delta_A,\delta_B$ are spectral with spectrum $\Lambda_A$ and $\Lambda_B$ respectively in (\ref{LambdaA}) and (\ref{LambdaB}). 
We will show that
$$
L := \Lambda+(\Lambda_B\times\Set{0})
$$
is a spectrum of $\eta$, where $\Lambda$ is a spectrum of $\mu\times\nu$ given in (\ref{eq-Lambda-1.5}) and $\Lambda_B$ is given by (\ref{LambdaB}).

This can be done by applying Lemma \ref{lm:conv} with $T=\Lambda_B\times\Set{0}$, $\varrho=\mu\times\nu$ and $D=B\times\{0\}$.  Condition (1) in Lemma \ref{lm:conv} is clearly satisfied.  (2) is also satisfied because $(A,B)$ is a complementary pair, so $((A+b)+[0,1])\cap ((A+b')+[0,1])$ intersects only at the boundary. Thus, $\mu ((A+b)+[0,1]\cap (A+b')+[0,1]) = 0$, so is the  $\varrho$ measure of its cartesian product with $\R$.  Finally,   it remains to verify that $L$ is orthogonal for $\eta$. Then we can apply Lemma \ref{lm:conv} to conclude that $\eta$ is spectral and the claim will be justified. 

Take $\lambda, \lambda' \in \Lambda$ and $\ell,\ell' \in \Lambda_B$ are given so that $\lambda+(\ell,0) \neq \lambda'+ (\ell', 0)$ and write $\lambda = (\lambda_1,\lambda_2)$, $\lambda' = (\lambda_1',\lambda_2')$ respectively. Let
$$
\Delta = \lambda+(\ell,0) - (\lambda'+ (\ell', 0)) = (\lambda_1-\lambda_1'+\ell-\ell', \lambda_2-\lambda_2')
$$
and note that 
\begin{equation}\label{eta-ft}
\ft\eta(\xi_1, \xi_2) =  \ft{\delta_B}(\xi_1) \ft\mu(\xi_1)  \ft\nu(\xi_2) = \ft{\mathds{1}_{[0, N]}}(\xi_1) \ft{\nu}(\xi_2) =  \ft{\delta_B}(\xi_1) \ft\varrho(\xi).
\end{equation}
We aim to show that $\ft\eta(\Delta) = 0$. This conclusion is immediate if  $\lambda=\lambda'$ (where $\widehat{\delta_B}(\ell-\ell')=0$) or $\ell = \ell'$ (where $\widehat{\varrho} (\lambda-\lambda')=0$).  Assume therefore that $\lambda \neq \lambda'$ and $\ell \neq \ell'$.

When $\lambda\ne\lambda'$, we have $\ft{\nu}(\lambda_2-\lambda_2')=0$ or $\ft{\mu}(\lambda_1-\lambda_1')=0$.  If $\ft\nu(\lambda_2-\lambda_2') = 0$ then $\widehat{\eta}(\Delta)=0$ in view of (\ref{eta-ft}). If $\ft\nu(\lambda_2-\lambda_2') \ne 0$,  we  must have $\ft{\mu}(\lambda_1-\lambda_1')=0$. We need to show that $\ft{\mathds{1}_{[0, N]}}(\lambda_1-\lambda_1'+\ell-\ell') = 0$. By (\ref{LambdaB}) and (\ref{eq-Lambda-1.5}),  $\lambda_1, \lambda_1', \ell$ and $\ell' \in \frac1N\ZZ$. The only way this can fail is if $\lambda_1-\lambda_1'+\ell-\ell'= 0$ which is equivalent to $\lambda_1-\lambda_1' = \ell'-\ell$. We have
$$
\lambda_1-\lambda_1' \in \Set{0} \cup {\mathcal Z}(\ft\mu) \text{ and }
\ell'-\ell \in \Set{0} \cup {\mathcal Z}(\ft{\delta_B}).
$$
From Lemma \ref{partitiion-lemma}, By the disjointness of ${\mathcal Z}(\ft{\delta_A})$ and ${\mathcal Z}(\ft{\delta_B})$, we also have the disjointness of the set ${\mathcal Z}(\ft\mu)$ and ${\mathcal Z}(\ft{\delta_B})$. This forces that  $\lambda_1=\lambda'_1$ and $\ell = \ell'$. But then $\ft{\mu}(\lambda_1-\lambda_1')\ne 0$. This results in a contradiction. Hence, $L$ must be a spectrum for $\eta$. The proof is therefore complete. 
\end{proof}

\section{Examples.}\label{sec4} We now illustrate some examples in which Theorem \ref{spectral} and Corollary \ref{cor-spectral} applies. Our theorem applies to all parallelepiped $P$ in $\R^d$. Indeed, we can write $P = A ([0,1]^d)$ for some invertible matrix $A$. Hence, ${\mathcal L}|_P\times \nu$ is spectral if and only if ${\mathcal L}|_{[0,1]^d}\times (\nu\circ A^{-1})$ is spectral. Hence, ${\mathcal L}|_P\times \nu$ is spectral implies $\nu$ must be spectral. The following gives an example for which the measure can be applied to some other general disconnected domains that may not tile $\N_n$ for any positive integers $n$. 

\begin{example}
Let $\Omega = [0,\frac13]\cup [\frac{5}{3},2]\cup [\frac{7}{3},\frac83]$. Then $\Omega$ is a fundamental domain of $\Z$. Its Fourier transform is given by 
$$
\widehat{\chi_{[0,\frac13]}} (\xi)\cdot \left(1+ e^{2\pi i \frac{5}{3}\xi}+e^{2\pi i \frac{7}{3}\xi}\right). 
$$
One can solve the zeros using the fact that $1+z_1+z_2=0$ for some $|z_1|=|z_2|=1$ if and only if $z_1,z_2$ is a permutation of $e^{\pm i2\pi/3}$. Its zero is exactly $\Z\setminus\{0\}$. Hence, by Theorem \ref{spectral}, ${\mathcal L}|_{\Omega}\times \nu$ is spectral if and only if $\nu$ is spectral. 
\end{example}


\begin{example}
 In this example, we construct a measure $\mu_0$ that has lattice Fourier zeros without full rank.    Let $\mu_0=\mathcal{L}|_{[0,1]^2}\ast \mathcal{L}|_{-[0,1]^2}$. It is well-known that 
\[
\mathcal{Z}(\widehat{\mathcal{L}|_{[0,1]^2}})=\left\{(\xi_1,\xi_2)\in\mathbb{R}^2\mid \xi_1\in \mathbb{Z}\setminus\{0\}\ \text{or}\ \xi_2\in \mathbb{Z}\setminus\{0\}\right\}.
\]
Note that \begin{equation}\label{0}
\widehat{\mu_0}(x)=|\widehat{\mathcal{L}|_{[0,1]^2}}(x)|^2\geq0,\ x\in \mathbb{R}^2,
\end{equation}
and $\mathbb{Z}^2\subseteq (\mathcal{Z}(\widehat{\mu_0})\cup\{0\})$.  Let 
\[A_1=\begin{pmatrix}
 1 &0 \\
 0 &1
\end{pmatrix},
A_2=\begin{pmatrix}
 1 &0 \\
 1 &1
\end{pmatrix},A_3=\begin{pmatrix}
 1 &1 \\
 0 &1
\end{pmatrix}
,A_4=\begin{pmatrix}
 1+\sqrt{2} &1 \\
\sqrt{2} &1
\end{pmatrix}.\]
For each $i = 1,2,3,4$, we define the measure $\mu_i(E) = \mu_0 (A_i (E))$, and let $S_i$ be the zero set of the Fourier transform $\mu_i$.  Then  we have
\begin{align}
S_1&=\{(x_1,x_2)\in \mathbb{R}^2:x_1\in \mathbb{Z}\setminus\{0\} \ \text{or}\ x_2\in \mathbb{Z}\setminus\{0\}\}\nonumber\\
S_2&=\{(x_1,x_2)\in \mathbb{R}^2:x_1-x_2\in \mathbb{Z}\setminus\{0\} \ \text{or}\ x_2\in \mathbb{Z}\setminus\{0\}\}\nonumber\\
S_3&=\{(x_1,x_2)\in \mathbb{R}^2:x_1\in \mathbb{Z}\setminus\{0\} \ \text{or}\ x_2-x_1\in \mathbb{Z}\setminus\{0\}\}\nonumber\\
S_4&=\{(x_1,x_2)\in \mathbb{R}^2:x_1-\sqrt{2}x_2\in \mathbb{Z}\setminus\{0\} \ \text{or}\ (1+\sqrt{2})x_2-x_1\in \mathbb{Z}\setminus\{0\}\}.\nonumber
\end{align}
In particular, $S_1\cap S_2\cap S_3 = \Z^2\setminus\{0\}$ and hence
\[S=S_1\cap S_2\cap S_3\cap S_4=(\mathbb{Z}^2\setminus\{0\})\cap S_4=\{(x_1,0):x_1\in \mathbb{Z}\setminus\{0\}\}.\]
Define $\mu(E)=\sum_{i=1}^{4}\frac{1}{4}\mu_0(A_i(E))$ for all Borel $E$,  
by \eqref{0}, we obtain
\[\widehat{\mu}(x)=\sum_{i=1}^{4}\frac{1}{4}\widehat{\mu_0}(A_i^{-T}x)=0 \ \text{iff} \ x\in S.\]
This shows that
\[\mathcal{Z}(\widehat{\mu})\cup\{0\}=\{(x_1,0):x_1\in \mathbb{Z}\}.\]
In this example the Fourier zeros of $\mu$ form the rank-one lattice $(\Z\setminus\{0\})\times \{0\}$ in $\R^2$. Indeed, $\mu$ is absolutely continuous with respect to the Lebesgue measure. $\mu$ cannot be a spectral measure with spectrum $\Z\times \{0\}$. In this case the conclusion of Theorem \ref{spectral} is false. This shows that, by the contrapositive of Theorem \ref{spectral}, $\mu\times \nu$ is never a spectral measure for all $\nu$.
\end{example}


\subsection{Failure of weak convergence in singular measures.} \label{3.1} We say that a sequence of uniformly discrete sets $\Lambda_n$ {\it converges weakly} to $\Lambda$ if for all  $\varepsilon>0$ and for all $R>0$, there exists $N\in\N$ such that for all $n>N$, 
$$
\Lambda_n\cap B(0,R)\subset \Lambda+B(0,\varepsilon), \ \Lambda\cap B(0,R)\subset \Lambda_n+B(0,\varepsilon).
$$
An important result about weak convergence is that if $\Lambda_n$ are spectra for a spectral set $\Omega$ and $\Lambda_n$ converges to $\Lambda$ weakly, then $\Lambda$ is also a spectrum for $\Omega$ (see e.g. \cite[Lemma 3.1]{GL-JFA}). We now illustrate that this is indeed false if we consider singular measures. 

\begin{example}
Let $\mu$ be the arc-length measure on the line segment $(0,0)$ to $(1,0)$ in $\R^2$. Let ${\bf v}_n = (1,n)^T$ and then $\Lambda_n = \{{\bf v}_n k: k\in\Z\}$ are spectra of the measure $\mu$.  For any $R>0$, for all sufficiently large $n$, $\Lambda_n\cap B(0,R)$ contains only the origin, so $\Lambda_n$ weakly converges to $\Lambda = \{0\}$, which is not a spectrum. 

More generally, if a continuous measure $\mu$ admits arbitrarily sparse spectra $\Lambda_n$  such that $0\in\Lambda_n$ and 
$$
\min\{|\lambda|: \lambda\in\Lambda_n\setminus \{0\}\}\to \infty, \ \mbox{as} \ n\to\infty,
$$
then $\Lambda_n$ converges weakly to $\Lambda = \{0\}$, which is a not a spectrum for $\mu$ (since $\mu$ is continuous, so $L^2(\mu)$ is infinite dimensional). By \cite{AL-2022}, any self-affine spectral measure admits such a sequence of spectra.  
\end{example}
The above example exploited the arbitrary sparseness to construct the counterexample whose weak limit can only be a singleton. One may ask if the weak convergence result can still hold if $\Lambda_n$ does not become arbitrarily sparse. In the following,  we give a non-trivial fractal example in $\R^1$ such that the spectra $\Lambda_n$ converges weakly to an infinite set $\Lambda$ which is mutually orthogonal but it is not a spectrum. 

\begin{example}
Let $\nu$ be the one-quarter Cantor measure defined in $(\ref{jdcd})$ in the introduction. Consider 
\[
\Lambda_n=\left\{3\cdot\sum_{j=0}^{n-1}4^{j}\varepsilon_j+\sum_{j=n}^{n+M}4^{j}\varepsilon_j: \varepsilon_j\in\{0,1\}, ~ M\ge 0\right\}.
\]
For any $k\geq1$, we define
\[\Lambda_{n,k}=
 \{0,3\}+4\{0,3\}+\cdots+4^{n-1}\{0,3\}+4^{n}\{0,1\}+\cdots+4^{n+k}\{0,1\}.\]
Then $\Lambda_n=\cup_{k=1}^{\infty} \Lambda_{n,k}$. We claim that all $\Lambda_n$ are spectra for $\nu$. Note that $\Lambda_n$ converges weakly to 
the countable set 
\[\Lambda=\{0,3\}+4\{0,3\}+\cdots+4^{n}\{0,3\}+\cdots.\]
However, it is a well-known fact that  $\Lambda$ is not a spectrum for $\nu$ \cite{dense}. Therefore, weak convergence fails to preserve the spectrality of the set in the limit.

It remains to see that $\Lambda_n$ is a spectrum for $\nu$. Indeed, by Theorem $4.4$ in \cite{DHL}, it suffices to show that
\begin{equation}\label{eq-show}
    \inf_{k\geq1} \inf_{\lambda\in \Lambda_{n,k}}\left|\hat{\nu}\left(\frac{\lambda}{4^{n+k+1}}\right)\right|^{2} >0.
\end{equation}

To this end, we notice that if $\lambda\in \Lambda_{n,k}$,
we have
\[0\leq \frac{\lambda}{4^{n+k+1}}\leq\frac{\sum_{j=0}^{n-1}4^{j}3+\sum_{j=0}^{k}4^{n+j}}{4^{n+k+1}}
=\frac{2}{3\cdot 4^{k+1}}+\frac{1}{3}-\frac{1}{4^{n+k+1}}<\frac{3}{8},\]
that is, for any $k\geq1$ and $\lambda\in \Lambda_{n,k}$, we have $0\leq\frac{\lambda}{4^{n+k+1}}<\frac{3}{8}$.
Hence, for any $k\geq1$, $\lambda\in \Lambda_{n,k}$, we have using the Fourier transform formula for $\nu$ (see e.g. \cite{dense}), 
\begin{align}
\left|\widehat{\nu}\left(\frac{\lambda}{4^{n+k+1}}\right)\right|
&=\prod_{j=1}^{\infty}\left|\cos\left(2\pi \frac{\lambda}{4^{n+k+1+j}}\right) \right|\nonumber \\
&\geq\prod_{j=1}^{\infty}\left(1-\frac{1}{2}\left(2\pi\frac{\lambda}{4^{n+k+1+j}}\right)^2 \right) ~ ~~~~~(\mbox{by} ~\cos x\ge 1-x^2/2) \nonumber \\
&>\prod_{j=1}^{\infty}\left(1-\frac{1}{2}\left(\frac{3\pi}{4^{j+1}}\right)^2 \right).\nonumber
\end{align}
By taking the infimum,  (\ref{eq-show}) holds.
Thus, $\Lambda_n$ is a spectrum for $\nu$. 
\end{example}

\bibliographystyle{plain}
\bibliography{ref}

\end{document}